\documentclass[english]{amsart}

\usepackage{babel}
\usepackage{amstext}
\usepackage{amsmath}
\usepackage{amssymb}
\usepackage{amsfonts}
\usepackage{latexsym}
\usepackage{ifthen}
\usepackage[all,2cell,dvips]{xy} \UseAllTwocells \SilentMatrices
\xyoption{all}
\newtheorem{teo}{Theorem}[section] 
\newtheorem{lem}[teo]{Lemma}

\newtheorem{cor}[teo]{Corollary}
\newtheorem{defi}[teo]{Definition}

\newtheorem{prop}[teo]{Proposition}

\title{Fundamental divisors on Fano varieties\\ of index $n-3$} 
\date{\today}

\author{Enrica Floris}

\address{Enrica Floris\\
IRMA, Universit\'e de Strasbourg et CNRS\\
7 rue Ren\'e-Descartes\\
67084 Strasbourg Cedex\\
France}
\email{floris@math.unistra.fr}

\begin{document}

\begin{abstract}

Let $X$ be a Fano manifold of dimension $n$ and index $n-3$.
Kawamata proved the non vanishing of
the global sections of the fundamental divisor in the case $n=4$.
Moreover he proved that if $Y$ is a general element of
the fundamental system then $Y$ has at most canonical singularities.
We prove a generalization of this result in arbitrary dimension.

\end{abstract}

\maketitle

\section{Introduction}
A Fano variety is an $n$-dimensional $\mathbb{Q}$-Gorenstein projective variety $X$
with ample anticanonical divisor $-K_X$.
The \textit{index} of a Fano variety $X$ is
$$i(X)=\sup\{t\in\mathbb{Q}\;|-K_X\sim_{\mathbb{Q}} tH,\: H\:\:\rm{ample,}\:\:\rm{Cartier} \}.$$
If $X$ has at most log terminal singularities,
the Picard group $Pic(X)$ is torsion free.
Therefore the Cartier divisor $H$ such that $-K_X\sim_{\mathbb{Q}} i(X)H$ is determined up to linear equivalence.
It is called the \textit{fundamental divisor} of $X$.

It is well known that $i(X)\leq n+1$. 
By the Kobayashi-Ochiai criterion, if $i(X)\geq n$ then $X$ is isomorphic to a hyperquadric or to a projective space.
Smooth Fano varieties of index $n-1$ have been classified by Fujita 
and smooth Fano varieties of index $n-2$ by Mukai.
An important tool in Mukai's classification is Mella's generalization ~\cite{Mella} of Shokurov's theorem ~\cite{Sho} 
on the general anticanonical divisor of smooth Fano threefolds: 
a general fundamental divisor $Y \in |H|$ is smooth.
\\
In this work we consider the case of Fano varieties of index $n-3$.
In ~\cite{Kaw} Kawamata proved, for a Fano variety $X$ of dimension 4, 
the non vanishing of the global sections of the anticanonical divisor.
Moreover he proved that if $Y$ is a general element of the anticanonical system
then $Y$ has at most canonical singularities.
In this work we propose the following generalization of Kawamata's result:

\begin{teo}\label{maintheorem1}
Let $X$ be a Fano variety of dimension $n\geq 4$ with at most Gorenstein canonical singularities
and index $n-3$, with $H$ fundamental divisor.

Suppose that $h^0(X,H)\neq 0$ and let $Y\in|H|$ be a general element.
Then $Y$ has at most canonical singularities.
\end{teo}

\begin{teo}\label{maintheorem2}
Let $X$ be a smooth Fano variety of dimension $n\geq 4$
and index $n-3$, with $H$ fundamental divisor.
\begin{enumerate}
\item If the dimension of $X$ is $n=4,\,5$, then $h^0(X,H)\geq n-2$.
\item If $n=6,\,7$ and the tangent bundle $T_X$ is $H$-semistable, then $h^0(X,H)\geq n-2$.
\item If the dimension of $X$ is $n\geq 8$, then $h^0(X,H)\geq n-2$.
\end{enumerate}
\end{teo}

If $X$ is a Fano manifold of dimension at least six and index $n-3$
then we know by Wi\'{s}niewski's work ~\cite{W2}, ~\cite{W3}
that either $X$ has a very special structure or the Picard number equals one.
In the latter case it is conjectured that the tangent bundle is always semistable.
Using Wi\'{s}niewski's result it should be possible to
deal with the cases $n=6,7$ and Picard number at least two
and our statement covers also the cases of dimension 6 and 7 and Picard number one
at least modulo the semistability conjecture.

Note that that if $n \geq 5$ and $Y \in |H|$ is a general element,
then by Theorem \ref{maintheorem1} and the adjunction formula
the variety $Y$ is a $(n-1)$-dimensional Fano variety with at most Gorenstein canonical singularities
and index $(n-1)-3$. Thus we can apply Theorem \ref{maintheorem1} inductively to construct a 
``ladder'' (cf. \cite{Ambro} for the terminology) 
\[
X \supsetneq Y_1 \supsetneq  Y_2 \supsetneq \ldots \supsetneq Y_{n-3}
\]
such that $Y_{i+1} \in |H|_{Y_i}|$ and $Y_{n-3}$ is a Calabi-Yau threefold with at most canonical singularities. This technique
reduces the study of Fano varieties of index $n-3$ to the fourfold case. In particular we obtain:

\begin{cor}
Suppose that we are in the situation of Theorem \ref{maintheorem2}. Then the base locus of $|H|$
has dimension at most two.
\end{cor}

This ``ladder technique" was used in ~\cite{HV}
to prove an integral version of the Hodge conjecture
for Fano manifolds of index $n-3$.\\
By Theorem \ref{maintheorem2} the condition $h^0(X,H)\geq n-2$ in ~\cite[Theorem 1.6]{HV} 
is satisfied
at least if $n>7$.
Thus the following is a corollary of ~\cite[Theorem 1.6]{HV} and Theorem \ref{maintheorem2}:
\begin{teo}
Let $X$ be a smooth Fano variety of dimension $n>7$ and index $n-3$.
Then the group $H^{2n-2}(X,\mathbb{Z})$ is generated over $\mathbb{Z}$
by classes of curves (equivalently, $Z^{2n-2}(X)=0$).
\end{teo}

\bigskip

\thanks{{\bfseries Acknowledgements.}
This article has been originally developed during my undergraduate thesis at the University of Paris VI
written under the supervision of Andreas H\"oring. 
I would like to thank Andreas H\"oring for all the things he taught me during this work, 
for constantly guiding and supporting me,
for all the remarks on several drafts of this paper.
He made me also notice the result of Hwang that improves the statement of
Theorem \ref{maintheorem2}.
I would also like to thank Claire Voisin and Andreas Horing
 for bringing to my attention the fact that Theorem \ref{maintheorem2},
 together with their result,
 leads to Theorem 1.4.
}

\section{Preliminaries}
We will work over $\mathbb{C}$ and
use the standard notation from ~\cite{H}.
In the following $\equiv$, $\sim$ and $\sim_{\mathbb{Q}}$ will respectively indicate
numerical, linear and $\mathbb{Q}$-linear equivalence of divisors.
The following definitions are taken from ~\cite{K} and ~\cite{KM}.

\begin{defi}%~\cite[p. 16]{K}
Let $(X,\Delta)$ be a pair, $\Delta=\sum a_i\Delta_i$ with $a_i \in\mathbb{Q}^{+}$. 
Suppose that $K_X+\Delta$ is $\mathbb{Q}$-Cartier.
Let $f\colon Y\rightarrow X$ be a birational morphism, $Y$ normal.
We can write
$$K_Y\equiv f^{\ast}(K_X+\Delta)+\sum a(E_i,X,\Delta) E_i.$$
where $E_i\subseteq Y$ are distinct prime divisors and $a(E_i,X,\Delta)\in\mathbb{R}$.
Furthermore we adopt the convention that a nonexceptional divisor $E$ appears in the sum
if and only if $E=f_{\ast}^{-1}D_i$ for some $i$
and then with coefficient $a(E,X,\Delta)=-a_i$.\\
The $a(E_i,X,\Delta)$ are called discrepancies.
\end{defi}

\begin{defi}%~\cite[p. 56]{KM}
We set
$$\rm{discrep}(X,\Delta)
=\inf \{a(E,X,\Delta)\;|\; E\, exceptional\, divisor\, over\, X\}.$$
A pair $(X,\Delta)$ is defined to be
\begin{itemize}
\item klt (kawamata log terminal) if $\rm{discrep}(X,\Delta)>-1$ and $\lfloor\Delta\rfloor= 0$,
\item plt (purely log terminal) if $\rm{discrep}(X,\Delta)> -1$,
\item lc (log canonical) if $\rm{discrep}(X,\Delta) \geq -1.$
\end{itemize}
\end{defi}

\begin{defi}
Let $(X,\Delta)$ be a klt pair, $D$ an effective $\mathbb{Q}$-Cartier $\mathbb{Q}$-divisor.
The log canonical threshold of $D$ for $(X,\Delta)$ is
$$\rm{lct}((X,\Delta),D)=\sup\{t\in\mathbb{R}^+|(X,\Delta+tD)\: is\: lc\}.$$
\end{defi}

\begin{defi}
Let $(X,\Delta)$ be a lc pair, $f\colon X'\rightarrow X$ a log resolution.
Let $E\subseteq X'$ be a divisor on $X'$ of discrepancy $-1$. Such a divisor is called a log canonical place. 
The image $f(E)$ is called center of log canonicity of the pair.
If we write $$K_{X'}\equiv \mu^{\ast}(K_X+\Delta)+E,$$
we can equivalently define a place as an irreducible component of $\lfloor -E \rfloor$.
We denote $CLC(X,\Delta)$ the set of all centers.
\end{defi}

\begin{defi}
Let $(X,\Delta)$ be a log canonical pair.
A minimal center for $(X,\Delta)$ is an element of $CLC(X,\Delta)$ that is minimal 
with respect to inclusion.
\end{defi}

\begin{defi}
Let $(X,\Delta)$ be a log canonical pair. A center $W$ is said to be exceptional
if there exists a log resolution $\mu\colon X'\rightarrow X$ for the pair $(X,\Delta)$
such that:
\begin{itemize}
\item there exists only one place $E_W\subseteq X'$ whose image in $X$ is $W$;
\item for every place $E'\neq E_W$, we have $\mu(E)\cap W=\emptyset$.
\end{itemize}
\end{defi}

\begin{lem} \label{c1/2} Let $X$ be a normal Gorenstein projective variety 
and $Y$ an effective Cartier divisor. 
Suppose that $(X,Y)$ is a non plt pair with and let $c$ be its log canonical threshold.
If all the minimal centers of $(X,cY)$ have codimension one, then $c\leq 1/2$.
\end{lem}
\begin{proof}
If $Y$ is not reduced, then $c\leq 1/2$.\\
If $Y$ is reduced then we claim that $(X,cY)$ has a minimal center of codimension at least two.
Suppose first that $c<1$. 
Since $Y$ is reduced 
the discrepancy of every $E\subseteq Y$ of codimension one equals to $c$. 
Then $E$ cannot be a center.
Suppose now that $c=1$. 
The pair $(X,Y)$ is not plt and $Y$ is reduced, so by definition there exists a center of codimension at least two.
\end{proof}

\begin{lem}~\cite[Lemma 5.1]{Ambro}\label{lognonpltinbaselocus}
Let $X$ be a normal variety and $\Delta$ a divisor on $X$ such that $(X,\Delta)$ is klt.
Let $H$ be an ample Cartier divisor on $X$ and $Y\in |H|$ a general element.
Suppose that $(X,\Delta+Y)$ is not plt and let $c$ be the log canonical threshold.
Then the union of all the centers of log canonicity of $(X,\Delta+cY)$
is contained in the base locus of $|H|$.
\end{lem}

\section{Non vanishing of the fundamental divisor}
Let $X$ be a Fano variety 
of dimension $n\geq 4$ and index $n-3$
with at most Gorenstein canonical singularities. Let $H$ be the fundamental divisor.
Then we have $-K_X\ \sim (n-3)H$.
We want to study the non vanishing of $h^0(X,H)$.
To do so, since by Kawamata-Viehweg vanishing theorem $\chi(X,H)=h^0(X,H)$,
we look for an ``explicit'' expression for $\chi(X,H)$.\\
For $j\in\{-1\ldots -(n-4)\}$ we have, by Kodaira vanishing and Serre duality,
\begin{eqnarray}\label{ann}
\chi(X,jH)&=(-1)^n h^n(X,jH)=(-1)^n h^0(X,-(n-3+j)H)=0.
\end{eqnarray}
Hence we can write
\begin{eqnarray}\label{abcd}
\chi(X,tH)=\frac{H^n}{n!}\prod_{j=1}^{n-4}(t+j)(t^4+at^3+bt^2+ct+d).
\end{eqnarray}
%and to achieve our purpose we need to calculate the coefficients $a,b,c,d$.
By Serre duality, the polynomial $\chi(X,tH)$ has the following symmetry property for every integer $t$:

\begin{eqnarray}\label{symm}
\chi(X,tH)&=(-1)^n\chi(X,K_X-tH)=(-1)^n\chi(X,-(n-3+t)H).
\end{eqnarray}

By ~\cite[Cor 1.4.4]{BCHM} there exists a birational map
$\mu\colon X'\rightarrow X$ where $X'$ has at most Gorenstein terminal singularities and $\mu^{\ast}K_X=K_{X'}$.
Since canonical singularities are rational
we have the equality
$$
\chi(X,tH)=\chi(X',t\mu^{\ast}H).
$$

Moreover, for a projective variety $X'$ with at most Gorenstein terminal singularities and $D$ a Cartier divisor on $X'$,
we have the following Riemann-Roch-formula
\begin{eqnarray}\label{rrterminal}
\chi(X',tD)=&\frac{D^n}{n!}t^n+\frac{-K_{X'} D^{n-1}}{2(n-1)!}t^{n-1}+\frac{(K_{X'}^2+c_2(X'))D^{n-2}}{12(n-2)!}t^{n-2}
\end{eqnarray}
\begin{eqnarray*}
&+p(t)+\chi(X',\mathcal{O}_{X'}),
\end{eqnarray*}
where $p(t)$ is a polynomial of degree $n-3$ and with no constant term.

By using the equalities \eqref{ann} and \eqref{symm} 
and applying \eqref{rrterminal} to $D=\mu^{\ast}H$
it is possible to compute the coefficients $a,b,c,d$ in \eqref{abcd} and obtain
$$
\begin{array}{ll}
a&=2(n-3)\\
&\\
c&=(n-3)(b-(n-3)^2)\\
&\\
b&=\frac{-n^4+8n^3+9n^2-160n+264}{24}+\frac{n(n-1)}{12}\frac{c_2(X)H^{n-2}}{H^n}\\
&\\
d&=\frac{n(n-1)(n-2)(n-3)\chi(X,\mathcal{O}_X)}{H^n}.
\end{array}
$$
Since $\chi(X,\mathcal{O}_X)=\chi(X',\mathcal{O}_{X'})=1$
for any Fano variety,
we obtain the following lemma.

\begin{lem}\label{formula}
Let $X$ be a Fano variety 
of dimension $n\geq 4$ and index $n-3$ with at most Gorenstein canonical singularities.
Let $H$ be a fundamental divisor.
Let $\mu\colon X'\rightarrow X$ be a birational morphism 
with $X'$ terminal and $K_{X'}=\mu^{\ast}K_X$.
Then
$$\chi(X,H)=\frac{H^n}{24}(-n^2+7n-8)+\frac{c_2(X')\mu^{\ast}H^{n-2}}{12}+n-3.$$
\end{lem}

\begin{prop}
Let $X$ be a Fano variety with at most Gorenstein canonical singularities.
Suppose that the dimension of $X$ is $n=4,\,5$, and the index $n-3$.
Then $h^0(X,H)\geq n-2$.
\end{prop}
\begin{proof}
We remark that if $n=4,\,5$ then $-n^2+7n-8>0$.
By ~\cite[Corollary 6.2]{KMM} we have that $c_2(X')\mu^{\ast}H^{n-2}\geq 0$.
Thus
$$
\begin{array}{rl}
h^0(X,H)&=\chi(X,H)=\frac{H^n}{24}(-n^2+7n-8)+\frac{c_2(X')\mu^{\ast}H^{n-2}}{12}+n-3\\
&\\
&\geq \frac{H^n}{24} (-n^2+7n-8)+n-3\geq n-2.\\
\end{array}
$$

\end{proof}

\begin{prop}\label{nonv_semist}
Let $X$ be a Fano variety of dimension $n\geq 4$ with at most Gorenstein canonical singularities
and index $n-3$, with $H$ fundamental divisor.
Suppose that the tangent bundle $T_X$ is $H$-semistable.
Then $h^0(X,H)\geq n-2$.
\end{prop}
\begin{proof}
If $T_X$ is $H$-semistable then $T_{X'}$ is $\mu^{\ast}H$-semistable.
Since $X'$ is terminal, its tangent bundle $T_{X'}$ is a Q-sheaf in codimension 2
and satisfies the conditions of the Bogomolov inequality  ~\cite[Lemma 6.5]{KMM}.
In our situation the inequality becomes
$$c_2(X')\mu^{\ast}H^{n-2}\geq \frac{n-1}{2n} c_1(X')^2 \mu^{\ast}H^{n-2}=\frac{(n-1)(n-3)^2}{2n}H^n.$$
We conclude by using it in the formula of Lemma \ref{formula}:
$$
\begin{array}{rl}
h^0(X,H)&=\frac{H^n}{24}(-n^2+7n-8)+\frac{c_2(X')\mu^{\ast}H^{n-2}}{12}+n-3\\
&\\
&\geq \frac{H^n}{24} \left[-n^2+7n-8+\frac{(n-1)(n-3)^2}{n}\right]+n-3\\
&\\
&=H^n\left(\frac{7n-9}{24}\right)+n-3\geq n-2.
\end{array}
$$
\end{proof}
\begin{proof} [Proof of Theorem \ref{maintheorem2}(3)]
If $X$ is a smooth Fano variety of Picard number one and index bigger than $(n+1)/2$
then by the result of ~\cite[Theorem 2.11]{Hwang}, ~\cite{HwangMok} the tangent bundle $T_X$ is $H$-semistable.
The numerical condition on the index is verified, if the index is $n-3$, for all $n\geq 8$.
Then the third statement of the theorem follows from
Proposition \ref{nonv_semist} since by Wisniewski's result ~\cite{W2}
a Fano variety with index $n-3$ and dimension at least 8
either has Picard number one
or is isomorphic to $\mathbb{P}^4\times\mathbb{P}^4$.
In the latter case the fundamental divisor is $H=p_1^{\ast}h+p_2^{\ast}h$
where $h$ is an hyperplane section on $\mathbb{P}^4$ and 
$p_1,p_2$ are the projections on the factors.
Then $h^0(\mathbb{P}^4\times\mathbb{P}^4,p_1^{\ast}h+p_2^{\ast}h)=h^0(\mathbb{P}^4,h)^2=25>n-2$.
\end{proof}

\section{Structure of a general element in $|H|$.}
Now we look for some result about the regularity of a general section $Y\in|H|$.
By inversion of adjunction \cite[Thm.7.5]{K} Theorem \ref{maintheorem1}
is equivalent to proving that the pair $(X, Y)$ is plt. This is the object of the following

\begin{prop}
Let $X$ be a Fano variety of dimension $n$ and index $n-3$ with at most Gorenstein canonical singularities.
Let $H$ be a fundamental divisor and suppose that
$h^0(X,H)\neq 0$. Let $Y\in|H|$ be a general element.
Then $(X,Y)$ is plt.
\end{prop}
\begin{proof}
We argue by contradiction and
suppose that $(X,Y)$ is not plt. %suppose n geq 5

Let $c$ be the log canonical threshold of $(X,Y)$.
By Lemma \ref{lognonpltinbaselocus} the pair $(X,cY)$ is plt in the complement of the base locus of $|H|$.
Since $(X,cY)$ is properly lc there exist a minimal center $W$.

By the perturbation technique ~\cite[Thm 8.7.1]{Kol07} we can find some rational numbers $c_1,c_2\ll 1$
and an effective $\mathbb{Q}$-divisor $A\sim_{\mathbb{Q}}c_1 H$
such that $W$ is exceptional for the pair $(X,(1-c_2) c Y+A)$.\\
By Kawamata's subadjunction formula ~\cite{Kawsub}
for every $\varepsilon >0$ there exists an effective $\mathbb{Q}$-divisor
$B_W$ on $W$ such that $(W, B_W)$ is a klt pair
and 
$$
\begin{array}{rl}
K_W + B_W&\sim_{\mathbb{Q}} (K_X+(1-c_2) cY+A+\varepsilon H)|_W\\
&\sim_{\mathbb{Q}} (-(n-3)+(1-c_2) c+c_1 +\varepsilon)H|_W.
\end{array}
$$
Set $\eta=-c_2 c+c_1 +\varepsilon$, then
$\eta$ is arbitrary small and
\begin{eqnarray}\label{formula1}
K_W + B_W\sim_{\mathbb{Q}}-(n-3-c-\eta)H|_W.
\end{eqnarray}
Let $Z$ be the union of all log canonical centers of the pair $(X,(1-c_2) c Y+A)$.
Let $\mathcal{I}_Z$ be the ideal sheaf of $Z$. We consider the exact sequence
$$0\rightarrow\mathcal{I}_Z (H)\rightarrow\mathcal{O}_X (H)\rightarrow\mathcal{O}_Z (H)\rightarrow 0.$$
By the Nadel vanishing theorem, 
$$H^1(X,\mathcal{I}_Z (H))=0.$$
Thus we obtain the short exact sequence
$$0\rightarrow H^0(X,\mathcal{I}_Z (H))\rightarrow H^0(X,\mathcal{O}_X (H))\rightarrow H^0(Z,\mathcal{O}_Z (H))\rightarrow 0.$$
By Lemma \ref{lognonpltinbaselocus} we know that $Z$ is contained in the base locus of $|H|$, so
$$H^0(X,\mathcal{I}_Z (H))\cong H^0(X,\mathcal{O}_X (H))$$
hence $h^0(Z,\mathcal{O}_Z (H))=0$.\\
Since $W$ is a connected component of $Z$, we have
$$h^0(W,\mathcal{O}_W (H))=0$$
If $W$ has dimension at most two, by ~\cite[Prop 4.1]{Kaw}
applied to $D=H|_W$
we obtain $h^0(W,\mathcal{O}_W (H))\neq0$, a contradiction.\\
If $\dim W$ is at least three,
then $(W,B_W)$ is log Fano of index $i(W)\geq n-3-c-\eta$.
Suppose that $$n-3-c-\eta>\dim W-3,$$
this implies $h^0(W,\mathcal{O}_W (H))\neq 0$ by 
~\cite[Theorem 5.1]{Kaw}, a contradiction.
Thus we are left with the case
 $$\dim W\geq n-c-\eta.$$
Since $c\leq 1$ and $\eta$ is arbitrary small, this implies $\dim W = n-1$.
This holds for all centers $W$,
then $c<1/2$ by Lemma \ref{c1/2}, 
thus we have $\dim W \geq n-1/2$, a contradiction.
\end{proof}

\end{document}